\documentclass[12pt]{amsart}
\usepackage{amsfonts}
\usepackage{amsmath}
\usepackage{amssymb}
\usepackage{a4}
\usepackage{multirow}
\usepackage{hyperref}

\newcommand{\heute}{29 August 2010}

\theoremstyle{plain}
\newtheorem{theorem}{Theorem}[section]

\newtheorem{lemma}[theorem]{Lemma}

\theoremstyle{remark}

\newtheorem{remark}[theorem]{Remark}

\newtheorem*{rk}{Remark}

\newcommand{\enref}[1]{\textup{(\ref{enum:#1})}}
\newcommand{\dashTwo}[1]{\textup{(\ref{two}${}'$)}}

\newcommand{\ignore}[1]{}

\newcommand{\Co}{\mathit{Co}_3}
\newcommand{\f}[1][p]{\mathbb{F}_{#1}}
\newcommand{\zz}{\mathbb{Z}}
\newcommand{\Gro}[1]{Gr\"ob\-ner}

\DeclareMathOperator{\Res}{Res}

\DeclareMathSymbol\normal{\mathrel}{AMSa}{"43}

\newcommand{\prank}[1][p]{\text{$#1$-rk}}

\newcommand{\abs}[1]{\left|#1\right|}
\newcommand{\coho}[2][p]{H^*(#2,\f[#1])}
\newcommand{\cohod}[3][p]{H^{#2}(#3,\f[#1])}
\newcommand{\HGH}[2]{{#2{\setminus}#1{/}#2}}

\begin{document}

\title[The mod-$2$ cohomology of Conway(3) is Cohen--Macaulay]{The Mod-$2$ Cohomology Ring of the Third Conway Group is Cohen--Macaulay}
\author[S. A. King]{Simon A. King}
\address{Mathematics Department \\ National University of Ireland, Galway \\
Ireland}
\email{simon.king@nuigalway.ie}
\author[D. J. Green]{David J. Green}
\address{Mathematical Institute \\
University of Jena \\ D-07737 Jena \\
Germany}
\email{david.green@uni-jena.de}
\author[G. Ellis]{Graham Ellis}
\address{Mathematics Department \\ National University of Ireland, Galway \\
Ireland}
\email{graham.ellis@nuigalway.ie}
\thanks{King was supported by Marie Curie grant MTKD-CT-2006-042685.
Green received travel assistance from DFG grants GR 1585/4-1 and GR 1585/4-2\@.}
\date{\heute}

\begin{abstract}
By explicit machine computation we obtain the mod-$2$ cohomology ring of the
third Conway group $\Co$. It is Cohen-Macaulay, has dimension~4, and is
detected on the
maximal elementary abelian $2$-subgroups.
\end{abstract}

\maketitle

\section{Introduction}
\noindent
There has been considerable work on the mod-2 cohomology rings of
the finite simple groups. Every finite simple group
of $2$-rank at most three has Cohen--Macaulay mod-2
cohomology~\cite{AdMi:Central}\@.
There are eight sporadic finite simple groups of $2$-rank four.
For six of these, the mod-$2$ cohomology ring has already been
determined, at least as a module over a polynomial
subalgebra~\cite[VIII.5]{AdMi:book2ed}\@.
In most cases the cohomology is not
Cohen--Macaulay. For instance, the Mathieu groups $M_{22}$~and $M_{23}$
each have maximal elementary abelian $2$-subgroups of ranks $3$ and $4$,
meaning that the cohomology cannot be Cohen--Macaulay:
see~\cite[p.~269]{AdMi:book2ed}\@.

The two outstanding cases have the largest
Sylow $2$-sub\-groups. The Higman--Sims group $\mathit{HS}$ has
size $2^9$ Sylow subgroup, and the cohomology of this 2-group
is known~\cite{ACKM}\@.
The third Conway group $\Co$ has
size $2^{10}$ Sylow subgroup.
\bigskip

\noindent
In this paper we consider~$\Co$. It stands out
for two reasons, one being that it has the largest Sylow $2$-subgroup.
The second reason requires a little explanation. The Mathieu group
$M_{12}$ has $2$-rank three.
Milgram observed
that $2$-locally it looks as if $M_{12}$ admits a faithful
representation in the Lie group $G_2$, but that is impossible.
Benson and Wilkerson made this more precise~\cite{BensonWilkerson}
by constructing a map of classifying spaces with good properties in
mod-2 cohomology.

Benson took a similar approach to~$\Co$.
After $2$-completion, its
classifying space admits a map to that of $\mathrm{DI}(4)$\@.
This is a monomorphism in mod-2 cohomology,
and $\coho[2]{\Co}$ is finitely generated as a module over its
image~\cite{Benson:Co3}\@.

The Dwyer--Wilkerson exotic finite loop space $\mathrm{DI}(4)$
has the rank four Dickson invariants as its mod-2
cohomology~\cite{DwyerWilkerson:DI4}\@. So Benson's result says
that the Dickson invariants form a homogeneous system of
parameters for $\coho[2]{\Co}$ in degrees $8, 12, 14, 15$.
Benson asks if these parameters
form a regular sequence~\cite{Benson:Solomon}\@.
That is, he suggests that $\coho[2]{\Co}$
might be Cohen--Macaulay.
Certainly the Dickson invariants constitute a filter-regular
system of parameters \cite[Thm 1.2]{Benson:DicksonCompCoho}\@.

By a mixture of machine computation and theoretical argument we obtain
the following theorem, answering Benson's question in the affirmative:

\begin{theorem}
\label{thm:main}
The mod-$2$ cohomology ring $\coho[2]{\Co}$ of the third Conway
group $\Co$ has the following properties:
\begin{enumerate}
\item
\label{enum:main1}
As a commutative $\f[2]$-algebra, it
has $16$ generators and $71$ relations. A full presentation
is given in Appendix~\ref{app:Pres}\@. The smallest generator degree is~$3$,
and the greatest is~$15$\@. The greatest degree of a relation
is $33$\@.
\item
\label{enum:main2}
It is Cohen--Macaulay, having Krull dimension $4$ and
depth~$4$\@.
\item
\label{enum:main3}
It has zero nilradical, and is detected on the maximal elementary abelian
$2$-subgroups. These all have rank~$4$, and form four conjugacy classes.
\item
\label{enum:main4}
Its Poincar\'e series is of the form
\[
P(t) = \frac{f(t)}{(1-t^8)(1-t^{12})(1-t^{14})(1-t^{15})},
\]
where $f(t) \in \zz[t]$ is the monic polynomial of degree $45$ with the
coefficients $1$, $1$, $1$, $1$, $2$, $3$, $3$, $4$, $4$, $6$, $7$, $8$,
$9$, $10$, $10$, $11$, $13$, $12$, $14$, $15$, $13$, $13$, $15$, $14$, $12$,
$13$, $11$, $10$, $10$, $9$, $8$, $7$, $6$, $4$, $4$, $3$, $3$,
$2$, $1$, $1$, $1$, $1$\@.
\end{enumerate}
\end{theorem}

\begin{remark}
Observe that the numerator $f(t)$ in the Poincar\'e series
is symmetric, in the sense that the coefficients remain the
same when read from back to front.
Benson--Carlson duality \cite[Thm~1.1]{BensonCarlson:Poincare}
forces this to happen, as
the cohomology ring is Cohen--Macaulay with parameters in degrees
$8$, $12$, $14$~and $15$.
\end{remark}

\noindent
We computed the cohomology of the
Sylow subgroup using our package~\cite{SimonsProg}\@. Then we computed
the stable elements degree by degree, following Holt~\cite{Holt:mechanical}\@.
We used our variant \cite[Thm~3.3]{128gps} of
Benson's test~\cite{Benson:DicksonCompCoho}
to tell when to stop.

\begin{remark}
We actually constructed Benson's Dickson invariants in
$\coho[2]{\Co}$, in
order to obtain an explicit filter regular system of parameters.
\end{remark}

\subsection*{Structure of the paper}
We recall the stable elements method in Section~\ref{section:stableTheory},
discussing how to reduce the number of stability checks.
In Section~\ref{section:compute} we consider how to implement
stability checks and Benson's test for non-$p$-groups. We highlight the
relevant group theory of $\Co$ in Section~\ref{sec:groups}, proving
Theorem~\ref{thm:main}\@.

\section{Stable elements}
\label{section:stableTheory}
\noindent
Let $p$ be a prime, $G$ a finite group, and $H \leq G$ a subgroup
whose index is coprime to~$p$.
Following Holt~\cite[p.~352]{Holt:mechanical} we
compute $\coho{G}$ as the ring
of stable elements
\cite[XII \S10]{CartanEilenberg}
in $\coho{H}$.
Recall that $x \in \coho{H}$ is stable if
\[
\forall g \in G \quad \Res^H_{H^g \cap H} (x) = g^* \Res^H_{H \cap {}^gH}(x)
\, , \quad \text{where $g^* = c_g^*$ for $c_g(h) = ghg^{-1}$.}
\]
Note that $H$ need not be a Sylow subgroup \cite[Prop.\@ 3.8.2]{Benson:I}\@.
The stability condition associated to~$g$
only depends on the double coset $HgH \in \HGH GH$.

\subsection*{Using intermediate subgroups}
Let $S$ be a Sylow $p$-subgroup of the finite group~$G$.
Holt observed that the total number
of stability conditions is reduced dramatically if one works up a tower
of subgroups
\[
S = G_0 \leq G_1 \leq G_2 \leq \cdots \leq G_n = G \, ,
\]
where each $\abs{G_i \colon G_{i-1}}$ is as small as possible.
One determines $\coho{G_i}$ as the ring
of stable elements  in $\coho{G_{i-1}}$.
Often we take $G_1 = N_G(Z(S))$.


\subsection*{Discarding double cosets}
For some double cosets the 
associated stability condition is satisfied
by every $x \in \coho{H}$. Such double cosets can be discarded.

For example, the trivial double coset $H1H$ can always be discarded.
And $HgH$ can be discarded if $H^g \cap H$ has order coprime to~$p$.
Proposition~18 of~\cite{J4} generalizes to a
group-theoretic criterion for the redundancy of some double cosets.

\begin{lemma}
\label{lemma:J4}
Let $H \leq G$ be a subgroup with $p'$ index.
Let $g \in G$, and let $T$~be a Sylow $p$-subgroup of $H^g \cap H$.
Suppose that transfer from $\coho{T}$~to $\coho{G}$ is the zero map.
Then the stability condition associated to $HgH$
is redundant.

In particular if there is a $p$-group $W \neq 1$ such that $T \times W \leq G$,
then the stability condition associated to $HgH$ is redundant.
\end{lemma}

\noindent
See Remark~\ref{remark:Co3stab} for an application of this result.

\begin{proof}
We do not claim
that the stability condition is always satisfied. 
The proof of the stable elements method in
\cite[Prop.\@ 3.8.2]{Benson:I} uses a weaker
condition: that stability holds after transfer from
$H^g \cap H$ to~$G$. So if the transfer map is zero, then the
double coset is redundant.  But transfer from $H^g \cap H$ factors
through transfer from $T$~to $G$, since transfer from $T$~to $H^g \cap H$
is a split surjection.

Last part:
Transfer from $T$~to $G$ factors through transfer from
$T$ to $T \times W$, which is zero: for restriction from $T \times W$ to
$T$ is a split surjection, and restriction followed by transfer is
multiplication by $\abs{W}$\@.
\end{proof}

\noindent
To perform the stability test for $HgH$ we first construct
the induced homomorphisms $\Res^H_{H^g \cap H}$ and $g^* \Res^H_{H \cap {}^gH}$,
determining the images of the ring generators. If each generator has the same image both
times then we discard the double coset. Similarly, we discard
it if the pair of maps has been seen already.
This too saves effort, for the most time-intensive step is the next one:
working out the matrices of the two linear maps
from $H^n(H,\f)$ to $H^n(H^g \cap H,\f)$ degree by degree.

\section{Computational aspects}
\label{section:compute}

\subsection*{Representing cohomology rings}
We consider how to represent
the cohomology ring of a finite group on the computer.
Reusing the results of previous computations saves time, but
it does involve coherence issues.

Let $G$ be a finite group and $S \leq G$ a Sylow $p$-subgroup.
We assume that we already know the cohomology of a group $\bar{S}$ isomorphic
to~$S$.
In order to make use of this computation we
choose an isomorphism $f \colon \bar{S} \rightarrow S$.
We can then store $\coho{G}$ by recording the map~$f$
together with the image ring $R_{G,f}$ given by
\[
R_{G,f} = f^* \left(\Res^G_S \coho{G} \right) \subseteq \coho{\bar{S}} \, .
\]

Now suppose that $\phi \colon G_1 \rightarrow G_2$ is a group homomorphism, and that
we calculated $\coho{G_i}$ for $i=1,2$ using the Sylow $p$-subgroup~$S_i$
and the isomorphism $f_i \colon \bar{S}_i \rightarrow S_i$.
We represent $\phi^*$ as the composition
\[
R_{G_2,f_2} \stackrel{\cong}{\longrightarrow}
\coho{G_2} \stackrel{\phi^*}{\longrightarrow} \coho{G_1}
\stackrel{\cong}{\longrightarrow} R_{G_1,f_1} \, .
\]
As $\phi(S_1)$ is a $p$-subgroup of~$G_2$, we may pick $g \in G_2$
such that $\phi'(S_1) \leq S_2$, where $\phi' = c_g \circ \phi$.
Then ${\phi'}^*  = \phi^*$, since $c_g$ is an inner automorphism of~$G_2$\@.
Let $\bar{\phi} \colon \bar{S}_1 \rightarrow \bar{S}_2$ be the
homomorphism $\bar{\phi} = f_2^{-1} \circ \phi' \circ f_1$.
Then $\bar{\phi}^*$ maps $R_{G_2,f_2} \subseteq \coho{\bar{S}_2}$
to $R_{G_1,f_1} \subseteq \coho{\bar{S}_1}$ in the desired way.

\subsection*{Stability and the representation}
\noindent
Let $S \leq H \leq G$, where $S$~is Sylow in $G$ and
$\coho{H}$ is known: so we know $R_{H,f}$
for an isomorphism $f \colon \bar{S} \rightarrow S$\@.
The stability test for $HgH$ asks for the equalizer of
$\phi_1^*, \phi_2^* \colon \coho{H} \rightarrow \coho{H^g \cap H}$,
where
$\phi_1, \phi_2 \colon H^g \cap H \rightarrow H$ are
$\phi_1(h) = h$ and $\phi_2(h) = ghg^{-1}$.

Typically the cohomology of $H^g\cap H$ will not yet be known,
but the cohomology of its Sylow subgroup~$T$ will be.
We have two options:
\begin{itemize}
\item
We compute $\coho{H^g \cap H}$ and
construct $\phi^*_1, \phi^*_2$ as above.
\item
We take the equalizer of $\psi_1^*, \psi_2^* \colon \coho{H} \rightarrow
\coho{T}$ instead, where $\psi_i = \phi_i\vert_T$.
This works since $\Res^{H^g \cap H}_T$ is injective.
\end{itemize}
To our surprise, the first method proved to be more efficient.
One possible explanation is that $H^n(H^g \cap H, \f)$ often
has considerably smaller dimension than $H^n(T,\f)$. This reduces the size
of the matrices representing the two maps: and matrix size seems to have the
greatest influence on running time.

\begin{rk}
Holt~\cite{Holt:mechanical} chooses good double coset representatives
at the outset. We are effectively taking the first ones we find and
correcting them later on.
\end{rk}

\subsection*{Computing stable elements degree by degree}
We have translated each stability check into taking the equalizer of
two known ring homomorphisms. We now have to determine the equalizers and
then take their intersection. One approach would be efficient algorithms
for ideals, though we might have to implement these ourselves.
Another would be to compute parameters for $\coho{G}$ using e.g.\@ Chern
classes, and then to use algorithms for noetherian modules.

We take a different approach and work degree by degree.
Then performing a stability check just means taking the
nullspace of a matrix. This is easier to implement, but
linear algebra on its own cannot tell when to stop.

Following Benson, we write $\tau_d \coho{G}$ for the $\f$-algebra generated
by the indecomposable elements of~$\coho{G}$ in degree${}\leq d$,
subject to the relations which hold in $\coho{G}$ in degrees${}\leq d$.
Assume that we already have $\tau_{d-1}H^\ast(G;\f[p])$,
and have recorded the image in $\coho{H}$ of each generator.
So we can construct the image in $\cohod{d}{H}$ of
each degree~$d$ standard monomial of $\tau_{d-1} \coho{G}$. A degree~$d$
relation in $\coho{G}$ corresponds to a linear dependence between these
images; and if the images do not span the subspace of stable elements
in $\cohod{d}{H}$, then we get new generators. This determines
$\tau_d \coho{G}$.

\begin{rk}
The third author has implemented the stable elements method 
in his HAP system. With Dutour Sikiri{\'c} he used it to compute the 
integral homology of the Mathieu group $M_{24}$ out to degree 
four~\cite{DutourSikiricEllis:Wythoff}.
\end{rk}

\subsection*{Constructing filter regular parameters}
We use Benson's test for completion \cite[Thm~10.1]{Benson:DicksonCompCoho}
to tell when $d$ is large enough to ensure
that $\tau_d \coho{G} = \coho{G}$.
The key step is to construct homogeneous elements
$h_1,\ldots,h_r \in \tau_d \coho{G}$ which are a filter-regular system
of parameters for both $\tau_d \coho{G}$ and $\coho{G}$. Here,
$r = \prank(G)$. We need one technical result.

\begin{lemma}
\label{lemma:fregSG}
Suppose that $c_1,\ldots,c_n \in \coho{G}$ is a filter-regular sequence
in $\coho{S}$.  Then it is filter-regular in $\coho{G}$ too.
\end{lemma}

\begin{proof}
$\coho{G}$ is a direct summand of the $\coho{G}$-module $\coho{S}$,
by virtue of the transfer map. The result follows.
\end{proof}

\noindent
Assume that $d$ is large enough, so that $\coho{G}$ is finite over
$\tau_d \coho{G}$.
By \cite[Coroll.~9.8]{Benson:DicksonCompCoho} there are filter-regular
parameters $d_1,\ldots,d_r$ which restrict to each maximal elementary
abelian $p$-subgroup as (powers of) the Dickson invariants.

Parameters in low degrees allow us to terminate the computation earlier.
The Dickson invariants are in rather high degree. Sections 2~and 3
of~\cite{128gps} present several ways of loweing the degrees.
One of these methods can however fail for non-$p$-groups: the weak
rank-restriction condition \cite[Lemma~2.3]{128gps}\@.
So we proceed as follows.
Set $z = \prank(Z(S))$.
\begin{enumerate}
\item
Construct the Dickson invariants $d_1,...,d_r$ if this is not too difficult.
\item
Using $d_1,\ldots,d_z$ or otherwise, find 
$c_1,\ldots,c_z \in \tau_d \coho{G}$ which restrict to parameters for
$\coho{Z(S)}$. For non-$p$-groups there is no guarantee that
these~$c_i$ may be chosen from among the ring generators.
\item
Using \cite[Lemma~2.3]{128gps},
extend $c_1,\ldots,c_z$ to
find filter-regular parameters $c_1,\ldots,c_r$ for $\coho{S}$,
extending $c_1,\ldots,c_z$.
If $c_{z+1},\ldots,c_r$ are stable then
$c_1,\ldots,c_r$ is filter-regular for $\coho{G}$ by Lemma~\ref{lemma:fregSG}\@.
\item
If only $c_r$ fails stability, then replace it by any stable class that
finishes off the parameter system.
\item
Use the factorization and nilpotent alteration methods
(Lemmas 2.5~and 2.7 of~\cite{128gps}) to reduce the degrees of
$d_1,\ldots,d_r$ and/or $c_1,\ldots,c_r$.
\end{enumerate}
With luck we thus construct a filter-regular system of parameters and can compute its
filter-degree type. Benson's test then gives us a degree bound involving the sum of
the parameter degrees. If this is too large then we use the existence result
\cite[Prop.~3.2]{128gps} for low-degree parameters over an extension field
in order to apply our variant of Benson's test \cite[Thm~3.3]{128gps}\@.

\section{The third Conway group}
\label{sec:groups}

\subsection*{The Sylow 2-subgroup}
The third Conway group $\Co$ is simple and admits a degree $276$ faithful
permutation representation~\cite{ATLAS}\@.
The Sylow $2$-subgroups have order~$2^{10}$.
The Online ATLAS~\cite{OnlineAtlas} contains explicit permutations for the
degree $276$ representation.
GAP~\cite{GAP4}
easily constructs the Sylow $2$-subgroup~$S$.

Despite its size, computing $\coho[2]{S}$ is a surprisingly
routine application of our program~\cite{SimonsProg}\@.
The result may be viewed online~\cite{SimonsWebsite}\@.
Duflot's lower bound for the depth \cite[Thm 12.3.3]{CarlsonTownsley} is one, and
the Krull dimension is four.
In fact the depth is three.
This led Dave Benson to reiterate to us his conjecture that
$\coho[2]{\Co}$ could be Cohen--Macaulay.

\subsection*{The maximal elementary abelian subgroups}
There are two conjugacy classes of involutions in~$\Co$:
classes 2A~and 2B with centralizer sizes 2,903,040 and
190,080 respectively.
Using GAP one sees that $\Co$ has four conjugacy classes
of maximal elementary abelian 2-subgroups. Each has
rank 4, and they are distinguished by the
number of 2A elements they contain. In ATLAS notation:
\begin{xalignat*}{4}
V_1 & = 2A_1B_{14} &
V_2 & = 2A_3B_{12} &
V_3 & = 2A_7B_8 &
V_4 & = 2A^4 \, .
\end{xalignat*}
For each $1 \leq r \leq 4$ there is a subgroup $2A^r \leq V_r$
containing all the 2A elements.

\begin{lemma}
\label{lemma:main2implies3}
In Theorem~\ref{thm:main}, \enref{main3} follows from \enref{main2}\@.
\end{lemma}

\begin{proof}
As $\coho[2]{\Co}$ has depth~$4$,
the centralizers of the rank four elementary abelians
detect the cohomology ring, by
a result of
Carlson \cite[Thm 12.5.2]{CarlsonTownsley}\@.
Using GAP one sees that these
elementary abelians are self-centralizing.
Elementary abelian $2$-groups have polynomial cohomology, so
the nilradical vanishes.
\end{proof}

\subsection*{A tower of subgroups}
The Sylow $2$-subgroup has 484,680 double cosets in~$\Co$.
It is therefore essential that we find a convenient tower of subgroups.

The order~$4$ elements in~$\Co$ form two conjugacy classes~\cite{ATLAS}\@.
Type 4A elements have size 23,040 centralizer,
and type 4B elements have size 1,536 centralizer.

\begin{lemma}
Let $S$ be a Sylow $2$-subgroup of $G = \Co$. Then
\begin{enumerate}
\item
The centre $Z(S)$ and the second centre $Z_2(S)$ have isomorphism types
$Z(S) \cong C_2$ and $Z_2(S) \cong C_4 \times C_2$.
\item
$Z_2(S)$ has Frattini subgroup $Z(S)$. So does each copy of $C_4$ in $Z_2(S)$.
\item
Precisely one subgroup $U \leq Z_2(S)$ is generated by a type $4A$ element.
\item
$N_G(Z_2(S)) \leq N_G(U) \leq N_G(Z(S))$.
\end{enumerate}
\end{lemma}

\begin{proof}
The first two are easily checked in GAP~\cite{GAP4} using the
permutation representation.
For the third statement one inspects the four order 4
elements in $Z_2(S) \cong C_4 \times C_2$, finding
two of type 4A, and two of type 4B\@. The centralizer sizes
differ, so the two type 4A elements lie in the same cyclic subgroup.

The last part now follows, for $Z(S)$ is a characteristic subgroup of~$U$,
and no other subgroup of $Z_2(S)$ is conjugate to~$U$ in $G = \Co$.
\end{proof}

\noindent
Consider the tower of subgroups
$S = G_0 \leq G_1 \leq G_2 \leq G_3 \leq G_4 = \Co$
given by
\begin{xalignat*}{3}
G_1 & = N_G(Z_2(S)) & G_2 & = N_G(U) & G_3 & = N_G(Z(S)) \, .
\end{xalignat*}
$G_3$ is a maximal subgroup of $\Co$~\cite{ATLAS}\@.
The sizes of the layers are as follows:
\[
\begin{array}{c|c|c}
i & \abs{G_i \colon G_{i-1}} & \abs{\HGH{G_i}{G_{i-1}}} \\
\hline
1 & 3 & 2 \\
2 & 15 & 3 \\
3 & 63 & 3 \\
4 & 170,775 & 7 \\
\hline
\end{array}
\]
As the trivial double coset can be discarded, working up the tower involves
a total of $1 + 2 + 2 + 6 = 11$ stability conditions.

\begin{remark}
\label{remark:Co3stab}
We can discard 4 more double cosets
when computing $\coho[2]{\Co}$ from $\coho[2]{G_3}$. 
Every maximal elementary abelian has rank 4, so
if the Sylow subgroup of
$G_3^g \cap G_3$ is elementary abelian of rank${}\leq 3$,
then Lemma~\ref{lemma:J4}
applies with $W = C_2$.
There are three
double cosets where the Sylow subgroup is elementary
abelian of order 4, and one where it is cyclic of order~$2$.
\end{remark}

\begin{proof}[Proof of Theorem~\ref{thm:main}]
We computed the mod-2 cohomology ring of the Sylow subgroup using
our package~\cite{SimonsProg}\@. We then used the stable elements
method and the computational methods of
Section~\ref{section:compute} to work up the tower of subgroups.

The depth is a by-product of a computation based on Benson's test.
The depth of $\coho[2]{G_i}$ is weakly increasing in~$i$: see
\cite[Thm 2.1]{Benson:NYJM2}, and note that
the proof only requires the index to be coprime to~$p$.

We remarked that $\coho[2]{G_0}$ already has depth~$3$.
It turns out that $\coho[2]{G_1}$ has depth~$4$.
So $\coho[2]{G_i}$ is Cohen--Macaulay for all $i \geq 1$.
Thus we established \enref{main1}, \enref{main2} and \enref{main4}\@.
So \enref{main3} holds too, by
Lemma~\ref{lemma:main2implies3}\@.
\end{proof}

\subsection*{Report on filter-regular parameters}
The 2-rank of $\Co$ is four, so the Dickson elements are in degrees 8, 12,
14 and 15 for any subgroup in the tower. As $\coho{\Co}$ contains
these Dickson invariants~\cite{Benson:Co3}, so does each $\coho{G_i}$:
no higher powers are necessary.

$G_0$ is the Sylow subgroup, with rank one centre. Applying
the weak rank-restriction condition \cite[Lemma~2.3]{128gps} we constructed
filter-regular parameters $c_1,c_2,c_3,c_4$ in degrees 8, 4, 6 and 7\@.
Using \cite[Prop.~3.2]{128gps} we demonstrated the existence
of filter-regular parameters in degrees 8, 4, 2 and 2\@.
This allowed us to terminate the calculation in degree 14, where the
last relation is found.

Our $c_1,c_2,c_3$ are stable for $G_3$, and so
$c_1,c_2,c_3$ is a filter-regular sequence in $\coho{G_i}$ for $i=1,2,3$.
For $i=1,2$ we found a fourth parameter in degree~$1$, so the
calculations for $\coho{G_1}$ and $\coho{G_2}$ terminate when
the last relation is found in degree 16\@.
For $G_3$ we found a fourth parameter in degree~7, detecting
completion in degree~21\@. The presentation is complete after
degree 18\@.

For $G_4 = \Co$ we had to construct Benson's Dickson
invariants, detecting completion in degree~45\@. The presentation
is complete after degree~33\@.

\appendix
\section{A Minimal  Ring Presentation}
\label{app:Pres}
\noindent
Ring generators are denoted by a letter with two indices.
$\coho[2]{\Co}$ has no nilradical. The letter `$b$' denotes a
generator with nilpotent restriction to the centre $Z(S)$
of the Sylow subgroup. The letter `$c$' denotes a \emph{Duflot element},
whose restriction to $Z(S)$ is non-nilpotent.
The first index gives the degree of the generator,
the second is to distinguish generators of the same degree.
This presentation is also available online~\cite{SimonsWebsite}\@.

A minimal generating set for $H^\ast(\Co;\f[2])$ is given by 
$$b_{4,0}, b_{6,1}, b_{8,1}, c_{8,3}, b_{12,1}, b_{12,7}, b_{14,1}, b_{3,0}, b_{5,0}, b_{7,0}, b_{7,1}, b_{9,0}, b_{11,5}, b_{13,1}, b_{13,7}, b_{15,13}.$$

The following polynomials form a minimal generating set of the relation ideal:
\begin{enumerate}
\item $b_{5,0}^2+b_{3,0} b_{7,0}+b_{4,0} b_{6,1}$
\item $b_{3,0}^2 b_{5,0}+b_{8,1} b_{3,0}+b_{4,0} b_{7,1}$
\item $b_{3,0} b_{9,0}+b_{4,0} b_{8,1}+b_{4,0}^3$
\item $b_{5,0} b_{7,0}+b_{3,0}^4+b_{6,1} b_{3,0}^2+b_{6,1}^2+b_{4,0}^3$
\item $b_{5,0} b_{7,1}+b_{4,0} b_{8,1}+b_{4,0}^3$
\item $b_{6,1} b_{7,1}+b_{4,0} b_{9,0}+b_{4,0} b_{6,1} b_{3,0}+b_{4,0}^2 b_{5,0}$
\item $b_{3,0}^2 b_{7,0}+b_{8,1} b_{5,0}+b_{4,0} b_{9,0}+b_{4,0} b_{6,1} b_{3,0}$
\item $b_{3,0}^2 b_{7,1}+b_{4,0} b_{9,0}+b_{4,0} b_{3,0}^3+b_{4,0}^2 b_{5,0}$
\item $b_{6,1} b_{3,0} b_{5,0}+b_{6,1} b_{8,1}+b_{4,0} b_{3,0} b_{7,1}+b_{4,0}^2 b_{3,0}^2+b_{4,0}^2 b_{6,1}$
\item $b_{5,0} b_{9,0}+b_{4,0} b_{3,0} b_{7,1}+b_{4,0} b_{3,0} b_{7,0}+b_{4,0}^2 b_{3,0}^2+b_{4,0}^2 b_{6,1}$
\item $b_{7,0} b_{7,1}+b_{4,0} b_{3,0} b_{7,0}$
\item $b_{6,1} b_{9,0}+b_{4,0} b_{6,1} b_{5,0}+b_{4,0}^2 b_{7,1}+b_{4,0}^3 b_{3,0}$
\item $b_{12,7} b_{3,0}+b_{4,0} b_{11,5}+b_{4,0} c_{8,3} b_{3,0}$
\item $b_{3,0}^5+b_{8,1} b_{7,0}+b_{6,1} b_{3,0}^3+b_{6,1}^2 b_{3,0}+b_{4,0}^2 b_{7,0}+b_{4,0}^3 b_{3,0}$
\item $b_{3,0} b_{13,1}+b_{8,1} b_{3,0} b_{5,0}+b_{8,1}^2+b_{4,0}^2 b_{8,1}$
\item $b_{3,0} b_{13,7}+b_{4,0} b_{12,7}+c_{8,3} b_{3,0} b_{5,0}$
\item $b_{5,0} b_{11,5}+b_{4,0} b_{12,7}+c_{8,3} b_{3,0} b_{5,0}$
\item $b_{7,0} b_{9,0}+b_{4,0} b_{3,0}^4+b_{4,0} b_{6,1} b_{3,0}^2+b_{4,0} b_{6,1}^2+b_{4,0}^4$
\item $b_{7,1} b_{9,0}+b_{8,1} b_{3,0} b_{5,0}+b_{8,1}^2+b_{4,0}^2 b_{3,0} b_{5,0}+b_{4,0}^2 b_{8,1}$
\item $b_{6,1} b_{11,5}+b_{4,0} b_{13,7}+b_{6,1} c_{8,3} b_{3,0}+b_{4,0} c_{8,3} b_{5,0}$
\item $b_{8,1} b_{9,0}+b_{4,0} b_{13,1}+b_{4,0} b_{8,1} b_{5,0}$
\item $b_{12,7} b_{5,0}+b_{4,0} b_{13,7}+b_{4,0} c_{8,3} b_{5,0}$
\item $b_{3,0}^2 b_{11,5}+b_{4,0} b_{13,7}+c_{8,3} b_{3,0}^3+b_{4,0} c_{8,3} b_{5,0}$
\item $b_{3,0} b_{7,1}^2+b_{4,0} b_{13,1}+b_{4,0}^2 b_{3,0}^3$
\item $b_{6,1} b_{12,7}+b_{4,0} b_{3,0} b_{11,5}+b_{4,0} c_{8,3} b_{3,0}^2$
\item $b_{3,0} b_{15,13}+b_{6,1} b_{12,1}+b_{4,0} b_{7,1}^2+b_{4,0} b_{7,0}^2+b_{4,0} b_{14,1}+b_{4,0}^3 b_{3,0}^2+c_{8,3} b_{3,0} b_{7,0}$
\item $b_{5,0} b_{13,1}+b_{4,0} b_{7,1}^2+b_{4,0}^3 b_{3,0}^2$
\item $b_{5,0} b_{13,7}+b_{4,0} b_{3,0} b_{11,5}+c_{8,3} b_{3,0} b_{7,0}+b_{4,0} c_{8,3} b_{3,0}^2+b_{4,0} b_{6,1} c_{8,3}$
\item $b_{7,0} b_{11,5}+c_{8,3} b_{3,0} b_{7,0}$
\item $b_{9,0}^2+b_{4,0} b_{7,1}^2+b_{4,0}^2 b_{3,0} b_{7,0}+b_{4,0}^3 b_{3,0}^2+b_{4,0}^3 b_{6,1}$
\item $b_{6,1} b_{13,1}+b_{4,0} b_{8,1} b_{7,1}+b_{4,0}^2 b_{8,1} b_{3,0}$
\item $b_{6,1} b_{13,7}+b_{4,0}^2 b_{11,5}+b_{6,1} c_{8,3} b_{5,0}+b_{4,0}^2 c_{8,3} b_{3,0}$
\item $b_{12,1} b_{7,0}$
\item $b_{12,7} b_{7,0}$
\item $b_{12,7} b_{7,1}+b_{8,1} b_{11,5}+b_{4,0}^2 b_{11,5}+b_{8,1} c_{8,3} b_{3,0}+b_{4,0}^2 c_{8,3} b_{3,0}$
\item $b_{14,1} b_{5,0}+b_{8,1}^2 b_{3,0}+b_{6,1} b_{8,1} b_{5,0}+b_{6,1}^2 b_{7,0}+b_{4,0} b_{15,13}+b_{4,0} b_{12,1} b_{3,0}+b_{4,0} b_{6,1} b_{3,0}^3+b_{4,0} b_{6,1}^2 b_{3,0}+b_{4,0}^2 b_{6,1} b_{5,0}+b_{4,0}^3 b_{7,1}+b_{4,0}^3 b_{7,0}+b_{4,0} c_{8,3} b_{7,0}$
\item $b_{14,1} b_{3,0}^2+b_{12,1} b_{3,0} b_{5,0}+b_{8,1} b_{3,0}^4+b_{6,1} b_{7,0}^2+b_{6,1} b_{14,1}+b_{6,1} b_{8,1} b_{3,0}^2+b_{6,1}^2 b_{8,1}+b_{4,0} b_{6,1} b_{3,0} b_{7,0}+b_{4,0}^2 b_{3,0}^4+b_{4,0}^2 b_{12,1}+b_{4,0}^2 b_{6,1} b_{3,0}^2+b_{4,0}^2 b_{6,1}^2+b_{4,0}^3 b_{8,1}+b_{4,0}^5$
\item $b_{5,0} b_{15,13}+b_{12,1} b_{3,0} b_{5,0}+b_{6,1} b_{7,0}^2+b_{6,1} b_{14,1}+b_{4,0} b_{8,1} b_{3,0} b_{5,0}+b_{4,0} b_{8,1}^2+b_{4,0}^3 b_{8,1}+c_{8,3} b_{3,0}^4+b_{6,1} c_{8,3} b_{3,0}^2+b_{6,1}^2 c_{8,3}+b_{4,0}^3 c_{8,3}$
\item $b_{7,0} b_{13,1}$
\item $b_{7,0} b_{13,7}+c_{8,3} b_{3,0}^4+b_{6,1} c_{8,3} b_{3,0}^2+b_{6,1}^2 c_{8,3}+b_{4,0}^3 c_{8,3}$
\item $b_{7,1} b_{13,7}+b_{8,1} b_{12,7}+b_{4,0}^2 b_{12,7}+b_{4,0} b_{8,1} c_{8,3}+b_{4,0}^3 c_{8,3}$
\item $b_{9,0} b_{11,5}+b_{8,1} b_{12,7}+b_{4,0}^2 b_{12,7}+b_{4,0} b_{8,1} c_{8,3}+b_{4,0}^3 c_{8,3}$
\item $b_{12,1} b_{3,0}^3+b_{6,1} b_{15,13}+b_{6,1} b_{12,1} b_{3,0}+b_{4,0} b_{14,1} b_{3,0}+b_{4,0} b_{12,1} b_{5,0}+b_{4,0} b_{8,1} b_{3,0}^3+b_{4,0} b_{6,1} b_{8,1} b_{3,0}+b_{4,0} b_{6,1}^2 b_{5,0}+b_{4,0}^2 b_{13,1}+b_{4,0}^2 b_{6,1} b_{7,0}+b_{4,0}^3 b_{3,0}^3+b_{4,0}^3 b_{6,1} b_{3,0}+b_{4,0}^4 b_{5,0}+b_{6,1} c_{8,3} b_{7,0}$
\item $b_{12,7} b_{9,0}+b_{8,1} b_{13,7}+b_{4,0}^2 b_{13,7}+b_{8,1} c_{8,3} b_{5,0}+b_{4,0}^2 c_{8,3} b_{5,0}$
\item $b_{3,0} b_{7,1} b_{11,5}+b_{8,1} b_{13,7}+b_{4,0}^2 b_{13,7}+b_{8,1} c_{8,3} b_{5,0}+b_{4,0} c_{8,3} b_{9,0}+b_{4,0} c_{8,3} b_{3,0}^3$
\item $b_{7,0}^3+b_{14,1} b_{7,0}$
\item $b_{7,1}^3+b_{8,1} b_{13,1}+b_{4,0}^2 b_{13,1}+b_{4,0}^3 b_{9,0}+b_{4,0}^3 b_{3,0}^3+b_{4,0}^4 b_{5,0}$
\item $b_{7,0} b_{15,13}+c_{8,3} b_{7,0}^2$
\item $b_{7,1} b_{15,13}+b_{12,1} b_{3,0} b_{7,1}+b_{8,1} b_{7,1}^2+b_{8,1} b_{14,1}+b_{8,1}^2 b_{3,0}^2+b_{6,1} b_{8,1}^2+b_{6,1}^2 b_{3,0} b_{7,0}+b_{4,0} b_{6,1} b_{3,0}^4+b_{4,0} b_{6,1}^2 b_{3,0}^2+b_{4,0}^2 b_{7,1}^2+b_{4,0}^2 b_{14,1}+b_{4,0}^2 b_{8,1} b_{3,0}^2+b_{4,0}^3 b_{3,0} b_{7,0}+b_{4,0}^4 b_{6,1}+b_{4,0} c_{8,3} b_{3,0} b_{7,0}$
\item $b_{9,0} b_{13,1}+b_{8,1} b_{7,1}^2+b_{4,0}^2 b_{7,1}^2+b_{4,0}^2 b_{8,1} b_{3,0}^2+b_{4,0}^4 b_{3,0}^2$
\item $b_{9,0} b_{13,7}+b_{4,0} b_{7,1} b_{11,5}+b_{4,0} c_{8,3} b_{3,0} b_{7,0}+b_{4,0}^2 c_{8,3} b_{3,0}^2+b_{4,0}^2 b_{6,1} c_{8,3}$
\item $b_{11,5}^2+b_{12,1} b_{3,0} b_{7,1}+b_{8,1} b_{7,1}^2+b_{8,1} b_{14,1}+b_{8,1}^2 b_{3,0}^2+b_{6,1} b_{8,1}^2+b_{6,1}^2 b_{3,0} b_{7,0}+b_{4,0} b_{6,1} b_{3,0}^4+b_{4,0} b_{6,1} b_{12,1}+b_{4,0} b_{6,1}^2 b_{3,0}^2+b_{4,0}^2 b_{7,0}^2+b_{4,0}^2 b_{8,1} b_{3,0}^2+b_{4,0}^3 b_{3,0} b_{7,0}+b_{4,0}^4 b_{3,0}^2+b_{4,0}^4 b_{6,1}+c_{8,3} b_{7,1}^2+b_{4,0}^2 c_{8,3} b_{3,0}^2+c_{8,3}^2 b_{3,0}^2$
\item $b_{12,7} b_{11,5}+b_{8,1} b_{15,13}+b_{6,1} b_{12,1} b_{5,0}+b_{4,0}^2 b_{12,1} b_{3,0}+b_{8,1} c_{8,3} b_{7,1}+b_{8,1} c_{8,3} b_{7,0}+b_{4,0} c_{8,3} b_{11,5}+b_{4,0} b_{8,1} c_{8,3} b_{3,0}+b_{4,0} c_{8,3}^2 b_{3,0}$
\item $b_{14,1} b_{9,0}+b_{8,1} b_{15,13}+b_{8,1}^2 b_{7,1}+b_{6,1} b_{12,1} b_{5,0}+b_{4,0} b_{12,1} b_{7,1}+b_{4,0} b_{6,1} b_{8,1} b_{5,0}+b_{4,0} b_{6,1}^2 b_{7,0}+b_{4,0}^2 b_{15,13}+b_{4,0}^2 b_{12,1} b_{3,0}+b_{4,0}^2 b_{6,1} b_{3,0}^3+b_{4,0}^2 b_{6,1}^2 b_{3,0}+b_{4,0}^3 b_{6,1} b_{5,0}+b_{4,0}^4 b_{7,1}+b_{4,0}^4 b_{7,0}+b_{8,1} c_{8,3} b_{7,0}+b_{4,0}^2 c_{8,3} b_{7,0}$
\item $b_{14,1} b_{3,0} b_{7,1}+b_{12,7}^2+b_{4,0} b_{7,1} b_{13,1}+b_{4,0} b_{8,1} b_{3,0}^4+b_{4,0} b_{8,1} b_{12,1}+b_{4,0} b_{6,1} b_{7,0}^2+{}$
$b_{4,0} b_{6,1} b_{14,1}+b_{4,0} b_{6,1} b_{8,1} b_{3,0}^2+b_{4,0} b_{6,1}^2 b_{8,1}+b_{4,0}^2 b_{8,1} b_{3,0} b_{5,0}+b_{4,0}^2 b_{8,1}^2+{}$ \\
$b_{4,0}^2 b_{6,1} b_{3,0} b_{7,0}+b_{4,0}^3 b_{3,0}^4+b_{4,0}^3 b_{6,1} b_{3,0}^2+b_{4,0}^3 b_{6,1}^2+b_{4,0}^6+b_{8,1} c_{8,3} b_{3,0} b_{5,0}+b_{8,1}^2 c_{8,3}+b_{4,0}^2 b_{8,1} c_{8,3}$
\item $b_{9,0} b_{15,13}+b_{12,7}^2+b_{4,0} b_{12,1} b_{3,0} b_{5,0}+b_{4,0} b_{6,1} b_{7,0}^2+b_{4,0} b_{6,1} b_{14,1}+b_{4,0}^2 b_{8,1} b_{3,0} b_{5,0}+b_{4,0}^2 b_{8,1}^2+b_{4,0}^4 b_{8,1}+b_{8,1} c_{8,3} b_{3,0} b_{5,0}+b_{8,1}^2 c_{8,3}+b_{4,0} c_{8,3} b_{3,0}^4+b_{4,0} b_{6,1} c_{8,3} b_{3,0}^2+b_{4,0} b_{6,1}^2 c_{8,3}+b_{4,0}^2 b_{8,1} c_{8,3}+b_{4,0}^4 c_{8,3}$
\item $b_{11,5} b_{13,7}+b_{12,7}^2+c_{8,3}^2 b_{3,0} b_{5,0}$
\item $b_{12,7} b_{13,7}+b_{4,0} b_{14,1} b_{7,1}+b_{4,0} b_{12,1} b_{9,0}+b_{4,0} b_{8,1} b_{13,1}+b_{4,0}^2 b_{14,1} b_{3,0}+b_{4,0}^2 b_{12,1} b_{5,0}+b_{4,0} c_{8,3} b_{13,7}+b_{4,0} c_{8,3} b_{13,1}+b_{4,0} c_{8,3}^2 b_{5,0}$
\item $b_{7,1}^2 b_{11,5}+b_{12,7} b_{13,1}+b_{4,0}^3 b_{13,7}+b_{4,0} c_{8,3} b_{13,1}+b_{4,0}^2 c_{8,3} b_{3,0}^3+b_{4,0}^3 c_{8,3} b_{5,0}$
\item $b_{11,5} b_{15,13}+b_{12,7} b_{14,1}+b_{12,1} b_{3,0} b_{11,5}+b_{8,1} b_{7,1} b_{11,5}+b_{4,0}^2 b_{7,1} b_{11,5}+b_{4,0}^3 b_{3,0} b_{11,5}+c_{8,3} b_{12,1} b_{3,0}^2+b_{6,1} c_{8,3} b_{12,1}+b_{4,0} c_{8,3} b_{7,0}^2+b_{4,0} c_{8,3} b_{14,1}+b_{4,0} b_{8,1} c_{8,3} b_{3,0}^2+{}$ \\
$b_{4,0}^2 c_{8,3} b_{3,0} b_{7,1}+b_{4,0}^3 c_{8,3} b_{3,0}^2+c_{8,3}^2 b_{3,0} b_{7,0}$
\item $b_{13,1}^2+b_{12,7} b_{14,1}+b_{12,1} b_{7,1}^2+b_{12,1} b_{3,0} b_{11,5}+b_{8,1} b_{7,1} b_{11,5}+b_{4,0}^2 b_{7,1} b_{11,5}+{}$ \\
$b_{4,0}^2 b_{12,1} b_{3,0}^2+b_{4,0}^3 b_{3,0} b_{11,5}+c_{8,3} b_{12,1} b_{3,0}^2+b_{4,0} c_{8,3} b_{7,1}^2+b_{4,0} b_{8,1} c_{8,3} b_{3,0}^2+{}$ \\
$b_{4,0}^2 c_{8,3} b_{3,0} b_{7,1}$
\item $b_{13,1} b_{13,7}+b_{8,1} b_{7,1} b_{11,5}+b_{4,0}^2 b_{7,1} b_{11,5}+b_{4,0}^3 b_{3,0} b_{11,5}+b_{4,0} b_{8,1} c_{8,3} b_{3,0}^2+{}$ \\
$b_{4,0}^2 c_{8,3} b_{3,0} b_{7,1}+b_{4,0}^3 c_{8,3} b_{3,0}^2$
\item $b_{13,7}^2+b_{4,0} b_{12,1} b_{3,0} b_{7,1}+b_{4,0} b_{8,1} b_{7,1}^2+b_{4,0} b_{8,1} b_{14,1}+b_{4,0} b_{8,1}^2 b_{3,0}^2+b_{4,0} b_{6,1} b_{8,1}^2+b_{4,0} b_{6,1}^2 b_{3,0} b_{7,0}+b_{4,0}^2 b_{6,1} b_{3,0}^4+b_{4,0}^2 b_{6,1} b_{12,1}+b_{4,0}^2 b_{6,1}^2 b_{3,0}^2+b_{4,0}^3 b_{7,0}^2+b_{4,0}^3 b_{8,1} b_{3,0}^2+b_{4,0}^4 b_{3,0} b_{7,0}+b_{4,0}^5 b_{3,0}^2+b_{4,0}^5 b_{6,1}+b_{4,0} c_{8,3} b_{7,1}^2+b_{4,0}^3 c_{8,3} b_{3,0}^2+c_{8,3}^2 b_{3,0} b_{7,0}+b_{4,0} b_{6,1} c_{8,3}^2$
\item $b_{14,1} b_{13,7}+b_{12,7} b_{15,13}+b_{8,1}^2 b_{11,5}+b_{4,0} b_{12,1} b_{11,5}+b_{6,1} b_{8,1} c_{8,3} b_{5,0}+b_{6,1}^2 c_{8,3} b_{7,0}+b_{4,0} c_{8,3} b_{15,13}+b_{4,0} b_{6,1} c_{8,3} b_{3,0}^3+b_{4,0} b_{6,1}^2 c_{8,3} b_{3,0}+b_{4,0}^2 b_{6,1} c_{8,3} b_{5,0}+b_{4,0}^3 c_{8,3} b_{7,1}+b_{4,0}^3 c_{8,3} b_{7,0}+b_{4,0} c_{8,3}^2 b_{7,0}$
\item $b_{7,1}^2 b_{13,1}+b_{12,7} b_{15,13}+b_{8,1} b_{12,1} b_{7,1}+b_{4,0} b_{6,1} b_{14,1} b_{3,0}+b_{4,0} b_{6,1} b_{12,1} b_{5,0}+{}$ \\
$b_{4,0} b_{6,1} b_{8,1} b_{3,0}^3+b_{4,0} b_{6,1}^2 b_{8,1} b_{3,0}+b_{4,0} b_{6,1}^3 b_{5,0}+b_{4,0}^2 b_{12,1} b_{7,1}+b_{4,0}^2 b_{6,1}^2 b_{7,0}+{}$ \\
$b_{4,0}^3 b_{15,13}+b_{4,0}^3 b_{12,1} b_{3,0}+b_{4,0}^3 b_{6,1} b_{3,0}^3+b_{4,0}^3 b_{6,1}^2 b_{3,0}+b_{4,0}^4 b_{6,1} b_{5,0}+b_{4,0}^3 c_{8,3} b_{7,0}$
\item $b_{14,1} b_{3,0} b_{11,5}+b_{8,1} b_{7,1} b_{13,1}+b_{8,1}^2 b_{12,1}+b_{4,0} b_{11,5} b_{13,1}+b_{4,0} b_{12,1} b_{12,7}+{}$ \\
$b_{4,0} b_{6,1} b_{12,1} b_{3,0}^2+b_{4,0}^2 b_{7,1} b_{13,1}+b_{4,0}^3 b_{8,1} b_{3,0} b_{5,0}+b_{4,0}^3 b_{8,1}^2+b_{4,0}^4 b_{12,1}+{}$ \\
$b_{4,0}^5 b_{8,1}+c_{8,3} b_{12,1} b_{3,0} b_{5,0}+b_{8,1} c_{8,3} b_{3,0}^4+b_{6,1} c_{8,3} b_{7,0}^2+b_{6,1} c_{8,3} b_{14,1}+b_{6,1} b_{8,1} c_{8,3} b_{3,0}^2+b_{6,1}^2 b_{8,1} c_{8,3}+b_{4,0} b_{8,1} c_{8,3} b_{3,0} b_{5,0}+b_{4,0} b_{8,1}^2 c_{8,3}+b_{4,0} b_{6,1} c_{8,3} b_{3,0} b_{7,0}+b_{4,0}^2 c_{8,3} b_{3,0}^4+b_{4,0}^2 c_{8,3} b_{12,1}+b_{4,0}^2 b_{6,1} c_{8,3} b_{3,0}^2+b_{4,0}^2 b_{6,1}^2 c_{8,3}+b_{4,0}^5 c_{8,3}$
\item $b_{13,1} b_{15,13}+b_{14,1} b_{7,1}^2+b_{8,1} b_{7,1} b_{13,1}+b_{8,1}^2 b_{12,1}+b_{4,0} b_{6,1} b_{12,1} b_{3,0}^2+b_{4,0}^2 b_{7,1} b_{13,1}+b_{4,0}^2 b_{12,1} b_{3,0} b_{5,0}+b_{4,0}^2 b_{8,1} b_{3,0}^4+b_{4,0}^2 b_{6,1} b_{7,0}^2+b_{4,0}^2 b_{6,1} b_{14,1}+b_{4,0}^2 b_{6,1} b_{8,1} b_{3,0}^2+{}$ \\
$b_{4,0}^2 b_{6,1}^2 b_{8,1}+b_{4,0}^3 b_{8,1} b_{3,0} b_{5,0}+b_{4,0}^3 b_{8,1}^2+b_{4,0}^3 b_{6,1} b_{3,0} b_{7,0}+b_{4,0}^4 b_{3,0}^4+b_{4,0}^4 b_{6,1} b_{3,0}^2+b_{4,0}^4 b_{6,1}^2+b_{4,0}^7$
\item $b_{13,7} b_{15,13}+b_{8,1} b_{7,1} b_{13,1}+b_{8,1}^2 b_{12,1}+b_{4,0} b_{6,1} b_{12,1} b_{3,0}^2+b_{4,0}^2 b_{7,1} b_{13,1}+{}$ \\
$b_{4,0}^3 b_{8,1} b_{3,0} b_{5,0}+b_{4,0}^3 b_{8,1}^2+b_{4,0}^4 b_{12,1}+b_{4,0}^5 b_{8,1}+c_{8,3} b_{12,1} b_{3,0} b_{5,0}+b_{6,1} c_{8,3} b_{7,0}^2+b_{6,1} c_{8,3} b_{14,1}+b_{4,0} b_{8,1} c_{8,3} b_{3,0} b_{5,0}+b_{4,0} b_{8,1}^2 c_{8,3}+b_{4,0}^3 b_{8,1} c_{8,3}+c_{8,3}^2 b_{3,0}^4+b_{6,1} c_{8,3}^2 b_{3,0}^2+b_{6,1}^2 c_{8,3}^2+b_{4,0}^3 c_{8,3}^2$
\item $b_{7,1} b_{11,5} b_{13,1}+b_{14,1}^2 b_{3,0}+b_{8,1} b_{12,1} b_{11,5}+b_{8,1}^3 b_{7,0}+b_{6,1} b_{8,1}^2 b_{3,0}^3+b_{6,1}^4 b_{7,0}+{}$ \\
$b_{4,0} b_{12,7} b_{15,13}+b_{4,0} b_{12,1}^2 b_{3,0}+b_{4,0} b_{8,1} b_{12,1} b_{7,1}+b_{4,0} b_{8,1}^2 b_{11,5}+b_{4,0} b_{6,1} b_{14,1} b_{7,0}+b_{4,0} b_{6,1} b_{8,1}^2 b_{5,0}+b_{4,0} b_{6,1}^2 b_{8,1} b_{7,0}+b_{4,0} b_{6,1}^3 b_{3,0}^3+b_{4,0} b_{6,1}^4 b_{3,0}+b_{4,0}^2 b_{8,1}^2 b_{7,0}+{}$ \\
$b_{4,0}^2 b_{6,1} b_{14,1} b_{3,0}+b_{4,0}^2 b_{6,1} b_{12,1} b_{5,0}+b_{4,0}^2 b_{6,1} b_{8,1} b_{3,0}^3+b_{4,0}^2 b_{6,1}^3 b_{5,0}+b_{4,0}^3 b_{12,1} b_{7,1}+b_{4,0}^3 b_{8,1}^2 b_{3,0}+b_{4,0}^4 b_{15,13}+b_{4,0}^4 b_{12,1} b_{3,0}+b_{4,0}^4 b_{8,1} b_{7,0}+b_{4,0}^4 b_{6,1} b_{3,0}^3+b_{4,0}^4 b_{6,1}^2 b_{3,0}+b_{4,0}^6 b_{7,1}+b_{8,1} c_{8,3} b_{15,13}+b_{8,1}^2 c_{8,3} b_{7,1}+b_{6,1} c_{8,3} b_{14,1} b_{3,0}+b_{6,1} b_{8,1} c_{8,3} b_{3,0}^3+{}$ \\
$b_{6,1}^2 b_{8,1} c_{8,3} b_{3,0}+b_{6,1}^3 c_{8,3} b_{5,0}+b_{4,0} c_{8,3} b_{12,1} b_{7,1}+b_{4,0} b_{6,1}^2 c_{8,3} b_{7,0}+b_{4,0}^2 c_{8,3} b_{15,13}+b_{4,0}^2 b_{6,1} c_{8,3} b_{3,0}^3+b_{4,0}^2 b_{6,1}^2 c_{8,3} b_{3,0}+b_{4,0}^3 b_{6,1} c_{8,3} b_{5,0}+b_{4,0}^4 c_{8,3} b_{7,0}+b_{8,1} c_{8,3}^2 b_{7,0}+b_{4,0}^2 c_{8,3}^2 b_{7,0}$
\item $b_{8,1} b_{11,5} b_{13,1}+b_{8,1} b_{12,1} b_{12,7}+b_{4,0} b_{14,1} b_{7,0}^2+b_{4,0} b_{14,1}^2+b_{4,0} b_{8,1} b_{7,1} b_{13,1}+b_{4,0}^2 b_{12,1}^2+b_{4,0}^3 b_{7,1} b_{13,1}+b_{4,0}^4 b_{8,1} b_{3,0} b_{5,0}+b_{4,0}^4 b_{8,1}^2+b_{4,0}^6 b_{8,1}+c_{8,3} b_{12,7}^2+b_{4,0} c_{8,3} b_{7,1} b_{13,1}+b_{4,0}^2 b_{8,1} c_{8,3} b_{3,0} b_{5,0}+b_{4,0}^2 b_{8,1}^2 c_{8,3}+b_{4,0}^4 b_{8,1} c_{8,3}+b_{8,1} c_{8,3}^2 b_{3,0} b_{5,0}+b_{8,1}^2 c_{8,3}^2+b_{4,0}^2 b_{8,1} c_{8,3}^2$
\item $b_{8,1} b_{12,7} b_{13,1}+b_{8,1} b_{12,1} b_{13,7}+b_{4,0} b_{14,1} b_{15,13}+b_{4,0} b_{12,1} b_{14,1} b_{3,0}+b_{4,0} b_{12,1}^2 b_{5,0}+b_{4,0} b_{8,1} b_{14,1} b_{7,1}+b_{4,0}^2 b_{14,1} b_{11,5}+b_{4,0}^2 b_{12,7} b_{13,1}+b_{4,0}^2 b_{12,1} b_{13,7}+b_{4,0}^2 b_{12,1} b_{13,1}+b_{4,0}^2 b_{8,1}^2 b_{3,0}^3+b_{4,0}^2 b_{6,1} b_{8,1}^2 b_{3,0}+b_{4,0}^2 b_{6,1}^2 b_{8,1} b_{5,0}+b_{4,0}^3 b_{14,1} b_{7,1}+b_{4,0}^3 b_{6,1} b_{15,13}+b_{4,0}^3 b_{6,1} b_{8,1} b_{7,0}+b_{4,0}^4 b_{12,1} b_{5,0}+b_{4,0}^4 b_{6,1}^2 b_{5,0}+b_{4,0}^5 b_{13,1}+b_{4,0}^5 b_{8,1} b_{5,0}+b_{4,0}^5 b_{6,1} b_{7,0}+b_{4,0}^6 b_{3,0}^3+b_{4,0}^6 b_{6,1} b_{3,0}+b_{4,0}^7 b_{5,0}+b_{4,0} c_{8,3} b_{14,1} b_{7,1}+b_{4,0} c_{8,3} b_{14,1} b_{7,0}+{}$ \\
$b_{4,0} b_{8,1} c_{8,3} b_{13,1}+b_{4,0} b_{6,1} c_{8,3} b_{12,1} b_{3,0}+b_{4,0}^3 b_{6,1} c_{8,3} b_{7,0}$
\end{enumerate}


\begin{thebibliography}{10}

\bibitem{ACKM}
A.~Adem, J.~F. Carlson, D.~B. Karagueuzian, and R.~J. Milgram.
\newblock The cohomology of the {S}ylow $2$-subgroup of the {H}igman-{S}ims
  group.
\newblock {\em J. Pure Appl. Algebra}, 164(3):275--305, 2001.

\bibitem{AdMi:Central}
A.~Adem and R.~J. Milgram.
\newblock The mod $2$ cohomology rings of rank $3$ simple groups are
  {C}ohen-{M}acaulay.
\newblock In F.~Quinn, editor, {\em Prospects in topology (Princeton, NJ,
  1994)}, Ann. of Math. Stud., vol. 138, pages 3--12. Princeton Univ. Press,
  Princeton, NJ, 1995.
\newblock arXiv:math/9503231v1 [math.AT].

\bibitem{AdMi:book2ed}
A.~Adem and R.~J. Milgram.
\newblock {\em Cohomology of finite groups}, volume 309 of {\em Grundlehren der
  Mathematischen Wissenschaften [Fundamental Principles of Mathematical
  Sciences]}.
\newblock Springer-Verlag, Berlin, second edition, 2004.

\bibitem{Benson:Co3}
D.~Benson.
\newblock Conway's group $\mathrm{Co}_3$ and the {D}ickson invariants.
\newblock {\em Manuscripta Math.}, 85(2):177--193, 1994.

\bibitem{Benson:NYJM2}
D.~Benson.
\newblock Modules with injective cohomology, and local duality for a finite
  group.
\newblock {\em New York J. Math.}, 7:201--215, 2001.

\bibitem{Benson:Solomon}
D.~J. Benson.
\newblock Cohomology of sporadic groups, finite loop spaces, and the {D}ickson
  invariants.
\newblock In P.~H. Kropholler, G.~A. Niblo, and R.~St{\"o}hr, editors, {\em
  Geometry and cohomology in group theory (Durham, 1994)}, volume 252 of {\em
  London Math. Soc. Lecture Note Ser.}, pages 10--23. Cambridge Univ. Press,
  Cambridge, 1998.

\bibitem{Benson:I}
D.~J. Benson.
\newblock {\em Representations and cohomology. {I}}.
\newblock Cambridge Studies in Advanced Math., vol.~30. Cambridge University
  Press, Cambridge, second edition, 1998.

\bibitem{Benson:DicksonCompCoho}
D.~J. Benson.
\newblock Dickson invariants, regularity and computation in group cohomology.
\newblock {\em Illinois J. Math.}, 48(1):171--197, 2004.

\bibitem{BensonCarlson:Poincare}
D.~J. Benson and J.~F. Carlson.
\newblock Projective resolutions and {P}oincar\'e duality complexes.
\newblock {\em Trans. Amer. Math. Soc.}, 342(2):447--488, 1994.

\bibitem{BensonWilkerson}
D.~J. Benson and C.~W. Wilkerson.
\newblock Finite simple groups and {D}ickson invariants.
\newblock In {\em Homotopy theory and its applications ({C}ocoyoc, 1993)},
  volume 188 of {\em Contemp. Math.}, pages 39--50. Amer. Math. Soc.,
  Providence, RI, 1995.

\bibitem{CarlsonTownsley}
J.~F. Carlson, L.~Townsley, L.~Valeri-Elizondo, and M.~Zhang.
\newblock {\em Cohomology Rings of Finite Groups}, vol.~3 of {\em Algebras
  and Applications}.
\newblock Kluwer Academic Publishers, Dordrecht, 2003.

\bibitem{CartanEilenberg}
H.~Cartan and S.~Eilenberg.
\newblock {\em Homological algebra}.
\newblock Princeton University Press, Princeton, N. J., 1956.

\bibitem{ATLAS}
J.~H. Conway, R.~T. Curtis, S.~P. Norton, R.~A. Parker, and R.~A. Wilson.
\newblock {\em Atlas of finite groups}.
\newblock Oxford University Press, Oxford, 1985.

\bibitem{DutourSikiricEllis:Wythoff}
M.~Dutour~Sikiri{\'c} and G.~Ellis.
\newblock {W}ythoff polytopes and low-dimensional homology of {M}athieu groups.
\newblock {\em J. Algebra}, 322(11):4143--4150, 2009.

\bibitem{DwyerWilkerson:DI4}
W.~G. Dwyer and C.~W. Wilkerson.
\newblock A new finite loop space at the prime two.
\newblock {\em J. Amer. Math. Soc.}, 6(1):37--64, 1993.

\bibitem{GAP4}
The GAP~Group.
\newblock {\em {GAP -- Groups, Algorithms, and Programming, Version 4.4.12}},
  2008.
\newblock \verb+(http://www.gap-system.org)+.

\bibitem{J4}
D.~J. Green.
\newblock On the cohomology of the sporadic simple group {$J\sb 4$}.
\newblock {\em Math. Proc. Cambridge Philos. Soc.}, 113(2):253--266, 1993.

\bibitem{128gps}
D.~J. Green and S.~A. King.
\newblock The computation of the cohomology rings of all groups of order $128$.
\newblock {\em J.~Algebra}, accepted. arXiv:1001.2577v2 [math.GR].

\bibitem{Holt:mechanical}
D.~F. Holt.
\newblock The mechanical computation of first and second cohomology groups.
\newblock {\em J. Symbolic Comput.}, 1(4):351--361, 1985.

\bibitem{SimonsWebsite}
S.~A. King.
\newblock Modular cohomology rings of finite groups.
\newblock Website, 2010.\\{\tt http://www.nuigalway.ie/maths/sk/Cohomology/}

\bibitem{SimonsProg}
S.~A. King and D.~J. Green.
\newblock {\em $p$-Group Cohomology Package}, 2009.
\newblock Peer-reviewed optional package for Sage~\cite{Sage}\@. \\ {\tt
  http://sage.math.washington.edu/home/SimonKing/Cohomology/}.

\bibitem{Sage}
W.~Stein et~al.
\newblock {\em {S}age {M}athematics {S}oftware ({V}ersion 4.2.1)}.
\newblock The Sage~Development Team, 2009.
\newblock {\tt http://www.sagemath.org}.

\bibitem{OnlineAtlas}
R.~A. Wilson et~al.
\newblock \emph{{ATLAS} of {F}inite {G}roup {R}epresentations -- {V}ersion 3}.
\newblock Website \verb+http://brauer.maths.qmul.ac.uk/Atlas/v3/+.

\end{thebibliography}

\end{document}